\documentclass[12pt]{amsart}

\usepackage{amsmath,amssymb,amsthm}
\usepackage{graphicx,epsfig}
\usepackage{enumerate}
\usepackage{mathtools}
\usepackage{xypic}
\usepackage{tikz}

\usepackage{multicol}
\usetikzlibrary{calc}
  \input{xy}
  \xyoption{all}
\usepackage{hyperref}
\usepackage{url}
\numberwithin{equation}{section}
\theoremstyle{plain}
\newtheorem{thm}{Theorem}[section]

\newtheorem{prop}[thm]{Proposition}
\newtheorem{lem}[thm]{Lemma}
\newtheorem{cor}[thm]{Corollary}

\newtheorem{conj}[thm]{Conjecture}

\theoremstyle{definition}
\newtheorem{dfn}[thm]{Definition}

\newtheorem{exmp}[thm]{Example}

\newtheorem{rem}[thm]{Remark}

\newtheorem{dfns-rems}[thm]{Definitions and Remarks}
\newtheorem{notas-rems}[thm]{Notations and Remarks}
\newtheorem{exmps-rems}[thm]{Examples and Remarks}

\def\0{{\bf 0}}

\def\m{\frak{m}}

\DeclareMathOperator{\supp}{Supp}
\DeclareMathOperator{\reg}{reg}

\DeclareMathOperator{\Ann}{Ann}

\DeclareMathOperator{\soc}{soc}

\begin{document}
\title{Lefschetz properties of monomial ideals with almost linear resolution}
\author{Nasrin Altafi}
\address{Department of Mathematics, KTH Royal Institute of Technology, S-100 44 Stockholm, Sweden}
\email{nasrinar@kth.se}

\author{Navid Nemati}
\address{Institut de Math\'{e}matiques de Jussieu. UPMC, 4 Place de Jussieu, 75005 Paris, France}
\email{navid.nemati@imj-prg.fr}

\begin{abstract}
We study the WLP and SLP of artinian monomial ideals in $S=\mathbb{K}[x_1,\dots ,x_n]$ via studying their minimal free resolutions. We study the Lefschetz properties of such ideals where the minimal free resolution of $S/I$ is linear for at least $n-2$ steps. We give an affirmative answer to a conjecture of Eisenbud, Huneke and Ulrich for artinian monomial ideals with almost linear resolutions.

\end{abstract}

\maketitle

\section{Introduction}
The weak Lefschtez property (WLP) of an artinian graded algebra $A$,  says that there exists a linear form $\ell$ that induces a multiplication map $\times \ell : A_i\longrightarrow A_{i+1}$ that has maximal rank for each $i$, i.e. that is either surjective or injective. The strong Lefschetz property (SLP)  says the map $\times \ell^t : A_i\longrightarrow A_{i+t}$ has maximal rank for each $i$ and $t$. It may seem a simple problem to establish the algebras with this properties but it turns out to be rather hard to determine even for natural families of algebras. It is also interesting to ask for which $t$ the map $\times \ell^t : A_i \longrightarrow A_{i+t}$ has maximal rank (see \cite{MN}). These fundamental properties have been studied by many authors from different point of views and for different families of algebras. In this article we study the Lefschetz properties of artinian monomial  ideals generated in  a single degree $d$ with assumptions on their minimal free resolutions. 

 In \cite{EHU}, Eisenbud, Huneke and Ulrich study the minimal free resolutions of artinian ideals in the polynomial ring $S=\mathbb{K}[x_1,\dots ,x_n]$. They proved that for an artinian ideal $I\subset S$ generated in degree $d$ with the minimal free resolution with  $p-1$ linear steps, we have that 
 $$\mathfrak{m}^d\subset I+(l_p+\cdots +l_p),$$
 where $l_1, \dots ,l_p$ are linearly independent linear forms and $\mathfrak{m}=(x_1,\dots ,x_n)$ is the maximal ideal of $S$ (see  \cite[Corollary 5.2]{EHU}). They also conjecture that under this assumption,  we have that $\mathfrak{m}^d\subset I+(l_p+\cdots +l_p)^2$ for sufficiently general linear forms $l_1\dots ,l_p$ (see \cite[Conjecture 5.5]{EHU}). Note that this conjecture in the case that $p=n-1$ is equivalent to the Lefschetz property of  $S/I$. This conjecture motivated us to study the Lefschetz properties of  artinian monomial ideals in $S=\mathbb{K}[x_1,\dots ,x_n]$ generated in degree $d$ by considering some assumptions on their minimal free resolutions.

As a corollary of the result of Eisenbud, Huneke and Ulrich \cite[Corollary 5.2]{EHU} we conclude that for artinian ideal $I\subset S$ generated in degree $d$ with almost linear resolution (the minimal free resolution of $S/I$ is linear for $n-1$ steps), $S/I$ satisfies the WLP. 
In Section \ref{section3}, we study the multiplication map by higher powers of a linear form on an artinian monomial algebra $ S/I$ where the minimal free resolution of $S/I$ is linear for $n-1$ steps, see Theorem \ref{xy, full rank}. In particular, in this case we are able to give an affirmative answer to the conjecture  posed by Eisenbud, Huneke and Ulrich . In the rest of this section, we prove that an artinian monomial algebra $S/I$ satisfies the SLP, where we consider an assumption on the generators of $I$, see Theorem \ref{SLP}.
For artinian monomial ideals $I\subset S$ where the minimal free resolution of $S/I$ is linear for $n-2$ steps the WLP does  not hold necessarily. In the main theorem of Section \ref{section4}, Theorem \ref{n-2 steps}, we prove for an artinian monomial ideal $I\subset S$ generated in degree $d$, $S/I$ satisfies the WLP if the minimal free resolution of $S/I$ has $n-2$ linear steps and $\m^{d+1}\subset I$.  The last assumption is equivalent to say Castelnuovo-Mumford regularity of $I$, $\reg(S/I)$, is $d$. Observe that the assumption on the Castelnuovo-Mumford regularity of $I$ is essential. In the polynomial ring with three variables the assumption of having $n-2$ linear steps in the minimal free resolution of $S/I$ is always fulfilled whenever $I$ is generated in a single degree, but  the Togliatti system defined by artinian monomial ideal $I=(x^3_1,x^3_2,x^3_3,x_1x_2x_3)$ fails the WLP.

\section{Preliminaries}
We consider standard graded algebra $S/I= \oplus_{i\geq 0} (S/I)_i$, where $S = \mathbb{K}[x_1,\dots , x_n]$ is a polynomial ring over a filed of characteristic zero and all $x_i$'s have degree $1$ and $I\subset S$ is an artinian homogeneous ideal generated in a single degree $d$. Let us define the weak and strong Lefschetz properties for artinian algebras.
\begin{dfn}
Let $I\subset S$ be an artinian homogeneous ideal. We say that $S/I$ has the \textit{weak Lefschetz property} (WLP) if there is a linear form $\ell \in (S/I)_1$ such that, for all integers $j$, the multiplication map 
$$
\times \ell : (S/I)_j\longrightarrow (S/I)_{j+1}
$$
has maximal rank, i.e. it is injective or surjective. In this case the linear form $\ell$ is called a \textit{Lefschetz element} of $S/I$. If for the general form $\ell \in (S/I)_1$ and for an integer number $j$ the map $\times \ell$ does not have the maximal rank we will say that the ideal $I$ fails the WLP in degree $j$. \\
We say that $S/I$ has the \textit{strong Lefschetz property} (SLP) if there is a linear form $\ell\in(S/I)_1$ such that, for all integers $j$ and $k$ the multiplication map 
$$
\times \ell^k : (S/I)_j\longrightarrow (S/I)_{j+k}
$$ has the maximal rank, i.e. it is injective or surjective. By abusing the notation we often say that $I$ satisfies or fails the SLP or WLP.
\end{dfn}
In the case of one variable, the WLP and SLP trivially hold since all ideals are principal. The case of two variables there is  a nice result in characteristic zero by Harima, Migliore, Nagel and Watanabe \cite[Proposition 4.4]{HMNW}.
\begin{prop}\label{two}
Every artinian ideal $I\subset \mathbb{K}[x,y]$ where $\mathbb{K}$ has characteristic zero, has the Strong Lefschetz property (and consequently also the Weak Lefschetz property). 
\end{prop}
In \cite[Proposition 2.2]{MMN}, Migliore, Mir\'o-Roig and Nagel by using the action of a torus on monomial algebras provide the existence of the canonical Lefschetz element.
\begin{prop}\cite[Proposition 2.2]{MMN}\label{mon} Let $I\subset S=\mathbb{K}[x_1,\dots ,x_n]$ be an artinian monomial ideal. Then $S/I$ has the weak Lefschetz property if and only if $x_1+x_2+\cdots +x_n$ is a Lefschetz element for $S/I$.
\end{prop}
\begin{rem}\label{Surjective}
Let multiplication map $\times \ell^c : (S/I)_{a-c}\rightarrow (S/I)_{a}$ where $I$ is an  ideal of $S$ generated in degree $d$. If  $H_{S/I}(d)\leq H_{S/I}(d-c)$, then  $\times \ell^c : (S/I){a-c}\rightarrow (S/I)_{a}$ has maximal rank for every $a$ if and only if it is surjective for $a=d$. In fact if the multiplication map  $\times \ell^c :(S/I)_{a-c}\rightarrow (S/I)_a$  is surjective we have that $\left [(S/I)/ \ell^c(S/I)\right ]_d=0$ therefore $\left [(S/I)/ \ell^c(S/I)\right ]_k$ for all $k\geq d$ and so   $\times \ell^c :(S/I)_{k-c}\rightarrow (S/I)_k$ is surjective for each $k\geq d$. On the other hand,  since $I$ is generated in degree $d$, the multiplication map by $\ell^c$ is injective in the degrees less than $d-c$ hence $S/I$.
\end{rem}
\subsection*{Minimal free resolution and a conjecture of Eisenbud, Huneke and Ulrich}
Let $S=\mathbb{K}[x_1, \dots , x_n]$ be a polynomial ring over a field $\mathbb{K}$ and $M$ be a finitely generated $S$-module. A minimal free resolution of $M$ is an exact sequence
\begin{center}
$0\rightarrow F_p\rightarrow F_{p-1}\rightarrow \cdots \rightarrow F_0\rightarrow M \rightarrow 0, $
\end{center}
where each $F_i$ is a graded $S$-free module of the form $F_i=\oplus S(-j)^{\beta_{i,j}(M)} $ such that the number of basis elements is minimal and each map is graded. The value ${\beta_{i,j}(M)}$ is called the $i$-th graded Betti numbers of $M$ of degree $j$. Note that the minimal free resolution of $M$ is unique up to isomorphism so the graded Betti numbers are uniquely determined. 
\begin{dfn}
Let $M$ be a finitely generated graded $S=k[x_1,\dots,x_n]$-module, then
$$
\reg(M):= \max_{i,j}\lbrace i-j \mid\beta_{i,j}(M)\neq 0\rbrace.
$$ 
Note that if $I$ is an artinian ideal, then $\reg(S/I)= \max_{i}\lbrace i\mid \m^{i+1}\subset I\rbrace$.
\end{dfn}
\begin{dfn}
Let $I\subset S= \mathbb{K}[x_1,\dots ,x_n]$ be an ideal of $S$ generated in degree $d$. We say that the minimal free resolution of $S/I$ is \textit{linear} for $r$ steps if $\beta_{i,i+j}(S/I)=0$, for all $1\leq i\leq r$ and all $j\geq d$. We say $S/I$ has  \textit{linear resolution} if $r=n$, and it has \textit{almost linear resolution} if $r=n-1$. 
\end{dfn}
Eisenbud, Huneke and Ulrich In  \cite{EHU} studied artinian  ideals in polynomial rings. They prove the following result:
\begin{thm}\cite[Corollary 5.2]{EHU}\label{thm- EHU}
Let $I\subset S$ be an artinian ideal generated in degree $d$ and $\mathfrak{m}=(x_1,\dots ,x_n)$. If the minimal free resolution of $I$ is linear for $p-1$ steps, then 
\begin{equation*}
\mathfrak{m}^d\subseteq I +(l_p, \dots , l_n)
\end{equation*} 
for linearly independent linear forms $l_p,l_{p+1},\dots ,l_n$. 
\end{thm}
The above result says that in terms of the minimal free resolution with these assumption we have $\reg( I +(l_p, \dots , l_n))\leq d$. They also conjecture that under the same assumptions as Theorem \ref{thm- EHU},  $\reg( I +(l_p, \dots , l_n)^2)\leq d$ where $l_p,\dots ,l_n$ are sufficiently general linear forms.
\begin{conj}\cite[Conjecture 5.4]{EHU}\label{conjecture}
Suppose $I\subset S$ is artinian ideal generated in degree $d$ and its minimal free resolution is linear for $p-1$ steps then 
\begin{equation*}
\mathfrak{m}^d\subseteq I +(l_p,,l_{p+1}, \dots , l_n)^2
\end{equation*} for sufficiently general linear forms $l_p,\dots ,l_n$. 
\end{conj}
\begin{rem}
Note that Theorem \ref{thm- EHU} holds for any set of  linearly independent linear forms $l_p,l_{p+1}\dots ,l_n$ but this is not the case necessarily for Conjecture \ref{conjecture}. For instance let $S=\mathbb{K}[x,y,z]$ and   $I=(x^3, y^3, z^3, xy^2, x^2y, xz^2, x^2z, y^2z, yz^2)$. The minimal free  resolution of  $I$ is as follows:
$$
0\rightarrow S(-5)^3\oplus S(-6)\rightarrow S(-4)^{12}\rightarrow S(-3)^9\rightarrow S\rightarrow 0
$$
and $I$ has  almost linear resolution. By Theorem \ref{thm- EHU} we have
$$
\mathfrak{m}^3\subseteq I+ (x).
$$
The statement of Conjecture \ref{conjecture} does not hold for the linear form $l=x$ since we have that $xyz\notin I+(x^2)$, $\mathfrak{m}^3\nsubseteq I+(x)^2$. But if $l= x-y$, one can check that $\mathfrak{m}^3\subseteq I+(l)^2$ and the Conjecture \ref{conjecture} holds in this case. Thus we need to consider sufficiently general linear forms in the conjecture.
\end{rem}
As a consequence of Theorem \ref{thm- EHU} we have:
\begin{cor}\label{almost wlp}
If $I\subset S$ is an artinian ideal generated in degree $d$ with almost linear resolution, then $S/I$ satisfies the WLP.
\end{cor}
\begin{proof}
Since the minimal free resolution of $I$ is linear for $n-1$ steps, Theorem \ref{thm- EHU} implies that for a general linear form $\ell$,  we have that $\mathfrak{m}^d \subset I+ (\ell)$. This is equivalent to have the surjective map $\times \ell :(S/I)_{d-1}\rightarrow (S/I)_d$ and the assertion follows from Remark \ref{Surjective}.
\end{proof}
\subsection*{Macaulay inverse systems}
Let us now recall some facts of the theory of the \textit{inverse system}, or \textit{Macaulay duality}, which will be a fundamental tool in this paper. For a complete introduction, we refer the reader to \cite{G} and \cite{IK}.

Let $R= \mathbb{K}[y_1,\dots ,y_n]$, and consider $R$ as a graded $S$-module where the action of $x_i$ on $R$ is partial differentiation with respect to $y_i$.\\
There is a one-to-one correspondence between graded artinian algebras $S/I$ and finitely generated graded $S$-submodules $M$ of $R$, where $I=\Ann_S(M)$  is the annihilator of $M$ in $S$ and, conversely, $M=I^{-1}$ is the $S$-submodule of $R$ which is annihilated by $I$ (cf. \cite[Remark 1]{G}), p.17).   
By duality, the map $\circ \ell : R_{i+1}\longrightarrow R_i$ is dual to the map  $\times \ell : (S/I)_{i}\longrightarrow (S/I)_{i+1}$. So the injectivity(resp. surjectivity) of the first map is equivalent to the surjectivity (resp. injectivity) of the second one. Here by $"\circ \ell"$ we mean the linear form $\ell$ acts on $R$.

The inverse system module $I^{-1}$ of an ideal $I$ is generated by monomials in $R$ if and only if $I$ is a monomial ideal in $S$. 
\section{Lefschetz properties of monomial ideals with $n-1$ linear steps }\label{section3}
The goal of this section is to give an affirmative answer to the  Conjecture \ref{conjecture} in the case of monomial ideals with almost linear resolution.\\
Let $ I\subset S$ be an artinian monomial ideal generated in degree $d$, in the following proposition we provide an upper bound for the Hilbert function  $H_{S/I}(d):=\dim_k(S/I)_d$ in terms of the number of linear steps in the minimal free resolution of $S/I$.
\begin{prop}\label{cobasis almost linear}
Let $I\subset S = \mathbb{K}[x_1,\dots  ,x_n]$ be an artinian monomial ideal generated in degree $d$. If the minimal free resolution of $I$ is linear for $r$ steps, then for every monomial $m\in Mon(S/I)_d$ we have $|\supp(m)|\geq r+1$. In particular, 
$$H_{S/I}(d)\leq \binom{n}{r+1}H_{S/I}(d-r-1) .$$
\end{prop}
\begin{proof}
Let 
 $$
 0\rightarrow F_n\xrightarrow{\varphi_{n}}F_{n-1}\xrightarrow{\varphi_{n-1}}\cdots F_1\xrightarrow{\varphi_1} F_0\xrightarrow{\varphi_0} 0
 $$
 be the minimal free resolution of $I^{-1}$ which is dual to the minimal free 
 resolution of $S/I$.  By Macaulay duality we can consider $m$ as an 
 element in $(I^{-1})_d$, there exists generator $m'\in I^{-1}$ such that 
 $m= h\circ m' $ for some $h\in S$. Suppose $|\supp (m)|\leq r$ and variables $y_1,\dots,y_{n- r}\notin \supp(m)$. If $\varphi_1(e_1)= m'$ for A basis element $e_1$ of $F_1$, then  $L_1 := (x_1h)\circ e_1$ is a first syzygy of $I^{-1}$. Therefore, it corresponds to a basis element of $F_2$, say $e_2$. Observe that,  $L_2 := x_2 \circ e_2$ is a second syzygy and it corresponds to a basis element of $F_3$. By continuing this procedure $n-r$ times, we  find a basis element for $F_{n-r}$ of degree higher than $d-n-r$. using the duality of the minimal free resolution of $I^{-1}$ and $S/I$, we get $\beta_{r, d+r}(S/I)\neq 0$ which contradicts the fact that the minimal free resolution of $S/I$ is linear for $r$ steps.
\end{proof}
\begin{rem}
In \cite[Proposition 11.1]{EHU} Eisenbud, Huneke and Ulrich find a lower bound for the number of generators of an ideal with almost linear resolution where the bound implies that  $H_{S/I}(d)\leq H_{S/I}(d-2)$. By Proposition \ref{cobasis almost linear} for a monomial ideal $I$ with almost linear resolution,  $H_{S/I}(d)\leq H_{S/I}(d-n)$.
\end{rem}
Let us define a specific class of well-known matrices with non-negative integer entries:
\begin{dfn}\label{Toeplitz definition}
For integers $n,m,k$ where $m\geq 1$, we define the following Toeplitz matrix $T_{n,m,k}$ as the $m\times m$ matrix 
$$T_{n,m,k} := 
\begin{bmatrix}
\binom{n}{k} & \binom{n}{k+1}   &  \binom{n}{k+2} & \cdots & \binom{n}{n}& 0& \cdots  & 0 \\\\
\binom{n}{k-1} &  \binom{n}{k}   &  \binom{n}{k+1}& \cdots  & \binom{n}{n-1}& \binom{n}{n} &\cdots  & 0 \\\\
\vdots&  \vdots &  \vdots & \vdots & \vdots  &\vdots&\vdots &\vdots  \\\\  

0 &  0   &0 &  \cdots &\binom{n}{k-3} & \binom{n}{k-2}  & \binom{n}{k-1} & \binom{n}{k}    
\end{bmatrix}
$$
where the $(i,j)^{th}$ entry of this matrix is $\binom{n}{k+j-i}$ and we use the convention that $\binom{n}{i}=0$
for $i\leq 0$ and $i>n$.
\end{dfn}
Determining the rank of such matrices is an open problem even in many specific cases. Here using the fact that any monomial algebra in the polynomial ring with two variables  has the SLP we show that Toeplitz matrix $T_{n,m,k}$ has maximal rank.
\begin{lem}\label{Toeplitz, lemma}
For every integers $0\leq k\leq n$ and $m\geq 1$, Toeplitz matrix $T_{n,m,k}$  is invertible.
\end{lem}

\begin{proof}
Consider ideal $I= (x^{m+n-k}, y^{m+k+1})$ in the polynomial ring $S=\mathbb{K}[x,y]$. Choose  monomial bases $\lbrace x^{m-i}y^{i-1}\rbrace_{i=1}^m$ and $\lbrace x^{m+n-k-j}y^{k+j-1}\rbrace _{j=1}^m$ for $k$-vector spaces $(S/I)_{m-1}$ and $(S/I)_{m+n-1}$, respectively. Observe that, $T_{n,m,k}$ is the matrix representing the multiplication map $\times(x+y)^n : (S/I)_{m-1}\rightarrow (S/I)_{m+n-1}$ with respect to the chosen monomial bases. Since by Proposition \ref{two}, any monomial $R$-algebra has the SLP, and by Proposition \ref{mon}, $x+y$ is a Lefschetz element for $S/I$, the multiplication map by $x+y$ is bijection and therefore Toeplitz matrix $T_{n,m,k}$ has nonzero determinant and therefore it is invertible.
\end{proof}
\begin{rem}
In \cite{14} there is a more general result about these Toeplitz matrices using the same technique as the proof of Lemma \ref{Toeplitz, lemma}.
\end{rem}
\begin{dfn}
Let $M= \lbrace m_1,\dots,m_r\rbrace$ be a set of monomials in $S=\mathbb{K}[x_1,\dots,x_n]$ of degree $d$. We say $M$ is a \textit{line segment} with respect to $(x_i,x_j)$ if
\begin{itemize}
\item[(1)] $ x_ix_j\vert m_t,  \hspace*{0.2 cm}\forall  \hspace*{0.2 cm}1\leq t\leq r$, 
\item[(2)]  $(x_j/x_i) m_t= m_{t+1}, \hspace*{0.2 cm} \forall\hspace*{0.2 cm} 1\leq t\leq r-1$.
\end{itemize}
In addition, for a monomial ideal $I\subset S$ generated in degree $d$ we say $M$ is a $S/I$-\textit{maximal line segment} with respect to $(x_i,x_j)$ if in addition we have $(x_i/x_j)m_1, (x_j/x_i)m_r\in I$.

\end{dfn}
\begin{lem}\label{MaximalSegment}
Let $M= \lbrace m_1,\dots,m_r\rbrace$ be a set of monomials of degree $d$ in $S=\mathbb{K}[x_1,\dots,x_n]$ which form a line segment w.r.t $(x_i,x_j)$ and let $J_M$ be the ideal generated by all the monomials in $S_d\setminus M$. If $x_i^ax_j^b\vert m_t$, for every $1\leq t\leq r$, then multiplication map $\times(x_i+x_j)^{a+b}: (S/J_M)_{k-a-b}\rightarrow (S/J_M)_{k}$ has maximal rank for every $k$.
\end{lem}
\begin{proof}
Without loss of generality assume $i=1$ and $j=2$. Since $H_{S/J_M}(d) \leq H_{S/J_M}(d-a-b)$ and $J_M$ is generated in degree $d$, by Remark \ref{Surjective} it is suffices to show that the map $\times (x_1+x_2)^{a+b}: (S/J_M)_{d-a-b}\rightarrow (S/J_M)_{d}$ is surjective.  Set $f_i:= m_i/x_1^ax_2^b$ and $\varphi$ the restriction of multiplication map $\times (x_1+x_2)^{a+b}:(S/J_M)_{d-a-b}\rightarrow (S/J_M)_{d}$ to $f_1,\dots,f_r$. Observe that the Toeplitz matrix $T_{a+b,r,a}$ is the matrix representing $\varphi$. By Lemma \ref{Toeplitz, lemma} this matrix is invertible so we can find preimage of each $m_i$ which means $\times (x_1+x_2)^{a+b}: (S/J_M)_{d-a-b}\rightarrow (S/J_M)_{d}$ is surjective.
\end{proof}

Using Lemma \ref{Toeplitz, lemma} and Lemma \ref{MaximalSegment} we prove that the multiplication map by a power of a linear form on an specific class of artinian monomial algebra has maximal rank in every degree.

\begin{thm}\label{xy, full rank}
Let $I\subset S=\mathbb{K}[x_1,\dots ,x_n]$ be an artinian monomial ideal generated in degree $d$. If there exist integers $1\leq i < j\leq n$ such that for every monomial $m\in (S/I)_d$,  $x_i^ax_j^b\vert m$ for some $a,b\geq 0$, then the multiplication map $\times(x_i+x_j)^{a+b}: (S/I)_{k-a-b}\rightarrow (S/I)_{k}$ has maximal rank for every $k$.
\end{thm}
\begin{proof}
Note that by Remark \ref{Surjective} it is enough to show that $\times (x_i+x_j)^{a+b}: (S/I)_{d-a-b}\rightarrow (S/I)_{d}$ is surjective. Without loss of generality we assume that $i=1$ and $j=2$. For an $n-2$-tuple $\textbf{a}:= (a_3, \dots a_n)\in (\mathbb{N}\cup \lbrace 0\rbrace)^{n-2}$ define 
$$
\mathcal{M}_{\textbf{a}}:= \lbrace x_1^{a_1}x_2^{a_2}x_3^{a_3}\dots x_n^{a_n}\in (S/I)_d\mid a_1,a_2 \geq 0 \rbrace.
$$
We will show that  $\mathcal{M}_{\textbf{a}}$ is in the image of $\times(x_1+x_2)^{a+b}:(S/I)_{d-a-b}\rightarrow (S/I)_{d}$ for every $\textbf{a}$. For a fixed $n-2$-tuple $\textbf{a}$,  $\mathcal{M}_{\textbf{a}}$ may contains different $S/I$-maximal line segments w.r.t $(x_1,x_2)$ by Lemma \ref{MaximalSegment} each of them  is in the image. By the procedure in the proof of Lemma \ref{MaximalSegment}, the  preimages of the elements in $\mathcal{M}_\mathfrak{a}$ are all distinct and this completes the proof.

\end{proof}
As a consequence of the above result and Proposition \ref{cobasis almost linear} we prove Conjecture \ref{conjecture} holds for monomial ideal $I\subset S$ with almost linear resolution.
\begin{thm}\label{ConjectureTrue}
Let $I\subset S=\mathbb{K}[x_1,\dots ,x_n]$ be an artinian monomial ideal generated in degree $d$ with almost linear resolution, then Conjecture \ref{conjecture} holds.
\end{thm}
\begin{proof}
Proposition \ref{cobasis almost linear} implies that for all $m\in Mon(S/I)_d$ we have $\vert \supp(m)\vert \geq n$ therefore $x_1\cdots x_n\vert m$. By Theorem \ref{xy, full rank} the multiplication map
$$
\times (x_i+x_j)^2: (S/I)_{d-2}\rightarrow (S/I)_d
$$
is surjective for every $1\leq i<j\leq n$. This implies that $\mathfrak{m}^d\subset I +(x_i+x_j)^2$ for every $1\leq i< j\leq n$.
\end{proof}
Now in the last theorem of this section we prove that for a class of artinian monomial ideals the SLP is satisfied.
\begin{thm}\label{SLP}
Let $I\subset S=\mathbb{K}[x_1,\dots ,x_n]$ be an artinian monomial ideal generated in degree $d$. If there exist integers $1\leq i < j\leq n$ such that for every monomial $m\in S_d$, $x_ix_j\vert m$ is equivalent to 	 $m\notin I$. Then $S/I$ enjoys the SLP.
\end{thm}
\begin{proof}
If $n=2$, by Proposition \ref{two} every artinian ideal $I$ has the SLP.  Let $n\geq 3$.
Without loss of generality, assume $i=1$ and $j=2$. Consider bigrading  $\deg(x_1)=\deg(x_2)=(1,0)$ and $\deg(x_i)=(0,1)$ for $3\leq i\leq n$ on $S$. By the assumption, if $b\geq  d$ we have $(S/I)_{(\ast , b)}\cong 0$. For every $b< d$ module $(S/I)_{(\ast,b)}$ is isomorphic to some copies of $ (k[x_1,x_2]/(x^{d-b}_1,x^{d-b}_2))_{\ast}$. Since every artinian algebra in two variables has the SLP, for all $a,b,c$ multiplication map $\times \ell^c: (S/I)_{(a-c,b)}\rightarrow (S/I)_{(a,b)}$ has maximal rank for a generic linear form $\ell$. 

For completing the proof it is sufficient to show that if $a+b=a'+b'$ then $\times \ell^c: (S/I)_{(a-c,b)}\rightarrow (S/I)_{(a,b)}$ is injective (respectively, surjective ) if and only if $\times \ell^c: (S/I)_{(a'-c,b')}\allowbreak \rightarrow (S/I)_{(a',b')}$  is injective (respectively, surjective). Since $(S/I)_{(\ast , b)}$ is a complete intersection artinian algebra, its Hilbert function (as a sequence) is symmetric and the maximum value obtained in the bidegree $(d-b-1,b)$. 
Now we have equivalent conditions:
\begin{align*}
&\times \ell^c: (S/I)_{(a-c,b)}\rightarrow (S/I)_{(a,b)} \hspace*{0.4 cm}\text{is}\hspace*{0.2 cm}\text{injective}\\
\Leftrightarrow &\vert (d-b-1)-(a-c)\vert \geq \vert (d-b-1)-a\vert \\
\Leftrightarrow &\vert (d-b'-1)-(a^\prime -c)\vert \geq \vert (d-b^\prime -1)-a^\prime\vert\\
\Leftrightarrow & \times \ell^c: (S/I)_{(a'-c,b')}\rightarrow (S/I)_{(a',b')} \hspace*{0.4 cm} \text{is}\hspace*{0.2 cm} \text{injective}.
\end{align*}
Similar argument works for surjectivity.
\end{proof}
We end this section by stating a conjecture that we have observed it  experimentally  in a large number of cases using Macaulay2 software \cite{M2}.
\begin{conj}\label{our conj}
Let $I\subset S=\mathbb{K}[x_1,\dots ,x_n]$ and $I\subset S$ be an artinian 
 monomial ideal generated in degree $d$. If for every monomial $m\in (S/I)_d$ we have  $x_1^{a_1}x_2^{a_2}\cdots x_n^{a_n}\vert m$, then the multiplication map $\times(\ell)^{a}: (S/I)_{k-a}\rightarrow (S/I)_{k}$ has maximal rank for every $k$, where $\ell=x_1+\cdots + x_n$ and $a= a_1+ \cdots +a_n$.
\end{conj}
If Conjecture \ref{our conj} is true by combining with \ref{cobasis almost linear} we get that if $I$ is a monomial ideal generated in a single degree $d$ with almost linear resolution then $\mathfrak{m}^d\subset I+(\ell)^n$ for a sufficiently general linear form $\ell$.
\section{Lefschetz properties via studying Macaulay inverse systems }\label{section4}
In this section we study the inverse system module $I^{-1}$ for monomial ideals in $S$ generated in degree $d$ and prove some results about the number of generators of $I$ satisfying the WLP. Also we study the artinian monomial ideals $I\subset S=\mathbb{K}[x_1,\dots ,x_n]$ generated in degree $d$ where the minimal free resolution of $S/I$ is linear for $n-2$ steps and we prove that if $\reg(S/I)=d$ then $S/I$ satisfies the WLP.

\begin{dfn}
In a polynomial ring $S=\mathbb{K}[x_1,\dots , x_n]$, for any monomial $m$ and variable $x_i$, define $$\deg_i(m):= max \lbrace e \mid {x^e_i}\vert m \rbrace .$$
\end{dfn}
\begin{prop}\label{proposition differential linear }
Let $I\subset S=k[x_1,\dots , x_n]$ be a monomial ideal of $S$ generated in a single degree $d$ and homogeneous form $F=\sum_{m\in (I^{-1})_d} a_m m\in (I^{-1})_d$  such that $ (x_1+\cdots +x_n)\circ F=0 $. If $a_m\neq 0$ and $y_i \vert  m$, then for all $0\leq j< \deg_i(m)$ there exists a monomial $m_{i,j}\in (I^{-1})_d$ with $\deg_i(m_{i,j})=j$ such that $a_{m_{i,j}}\neq 0$.
\end{prop}
\begin{proof}
Let $m= y_1^{b_1}\cdots y_n^{b_n}\in (I^{-1})_d$ and $a_m\neq 0$,
$$
(x_1+\cdots +x_n)\circ m = a_m \sum _{y_i\vert m} b_i \dfrac{m}{y_i}. \quad \text{}\quad  
$$
Since $(x_1+\cdots +x_n)\circ F=0$ for each $1\leq i\leq n$ where $y_i | m$  there exists a monomial $m^\prime\in (I^{-1})_d$ with nonzero coefficient in $F$ such that $\frac{m}{y_i}= \frac{m^\prime}{y_k}$ for some $ k\neq i$. Note that $\deg_i(m)=\deg_i(m^\prime)+1=b_i+1$ and define $m_{i,b_i-1} := m^\prime$.  If $b_i-1\neq 0$ then we can do the same and find $m_{i,b_i-2}$ in the support of $F$. The assertion follows by continuing this procedure to find distinct monomials $m_{i,b_i-3},\dots , m_{i,0}$ in the support of $F$.
\end{proof}

\begin{cor}
Let $I\subset S$ be a monomial ideal generated in degree $d$. If the multiplication map $\times \ell : (S/I)_{d-1}\rightarrow (S/I)_d$ fails to be surjective for every linear form $l$, then $H_{S/I}(d)\geq d+1$. In other word, if $H_{S/I}(d)\leq d$ then $S/I$ enjoys the WLP.
\end{cor}

\begin{proof}
Suppose the multiplication map $ \times (x_1+\cdots + x_n): (S/I)_{d-1}\rightarrow (S/I)_d$ is not surjective so by Macaulay duality there exists a non-zero form $F= \sum_{m\in (I^{-1})_d} a_m m\in(I^{-1})_d$ such that 
$(x_1+\cdots +x_n)\circ F=0$. 
Let $m= y_1^{b_1}\cdots y_n^{b_n}\in (I^{-1})_d $  be a monomial in the support of the form $F$, using Proposition \ref{proposition differential linear }, there are at least $\deg_1(m)+\cdots + \deg_n(m)= b_1+\cdots +b_n=d$ monomials different from $m$ with nonzero coefficients in  $F$. Therefore we have that $H_{S/I}(d)\geq d+1$.
\end{proof}
\begin{rem}
In \cite{14} the first author with Mats Boij provide a better bound for $H_{S/I}(d)$ when $I$ is an artinian monomial ideal in $S$ generated in degree $d$ and fails the WLP.
\end{rem}

\begin{dfn}
Let $I\subset S$ be an ideal, the \textit{socle} elements of $S/I $ is 
$$\soc(S/I)=\lbrace f\in S/I \mid \m f=0 \rbrace.$$
\end{dfn}
If $I\subset S$ is an artinian ideal with linear resolution it equals a power of maximal ideal  and therefore it satisfies the WLP trivially. On the other hand in Corollary \ref{almost wlp} we have seen that an artinian ideal $I\subset S$ with almost linear resolution satisfies the WLP. In the following result we determine whether an artinian monomial ideal $I\subset S$ has the WLP where the minimal free resolution of $S/I$ is linear for $n-2$. 
\begin{thm}\label{n-2 steps}
Let $I\subset S=\mathbb{K}[x_1,\dots ,x_n]$ be a monomial ideal generated in degree $d$ and $\mathfrak{m}^{d+1}\subset I$.  If the minimal free resolution of $S/I$ is linear for $n-2$ linear steps, then $S/I$ satisfies the WLP.
\end{thm}
\begin{proof}
Since $I$ has linear resolution for $n-2$ steps by Proposition \ref{cobasis almost linear}, for all $m\in \soc(S/I)$ we have $\vert\supp(m)\vert \geq n-1$. If for all $m\in \soc(S/I)$ we have $\vert\supp(m)\vert=n$ then clearly we have the WLP.  Suppose  that there exists $m\in \soc(S/I)$ such that $\vert\supp(m)\vert=n-1$.
Since $I$ is generated in degree $d$, to prove that $S/I$ has the WLP it is enough to  show that the multiplication map 
$$
\times (x_1+\cdots + x_n): (S/I)_{d-1}\rightarrow (S/I)_d
$$
is surjective, or equivalently the differentiation map 
$$
\circ (x_1+\cdots + x_n): (I^{-1})_d\rightarrow (I^{-1})_{d-1}
$$ is injective. Suppose not, so there exists a non-zero form $F= \sum_{m\in I^{-1}} a_m m$ such that 
\begin{equation}\label{partial}
(x_1+\cdots +x_n)\circ F=0 .
\end{equation}
Observe that, there exists a monomial $m\in I^{-1}$ with non-zero coefficient in $F$ such that  $\vert\supp(m)\vert=n-1$. Let $m= y^{b_1}_1, \dots , y^{b_{n-1}}_{n-1}$ hence
$$
(x_1+\cdots + x_n)\circ m= b_1\frac{ m }{ y_1}+ \cdots + b_{n-1}\frac{ m}{ y_{n-1}}.
$$
If Equation (\ref{partial}) holds there must exist $m_1\in I^{-1}$ and integer $1\leq i\leq n$ such that $\dfrac{m_1}{y_i}= \dfrac{m}{y_1}$. Suppose $y_n\nmid m_1$ then $\vert \supp(m_1)\vert =\vert \supp(m)\vert =n-1$. Let 
$$
F_1:=  (b_i+1)x_1\circ m - ( b_1)x_i\circ m_1,\hspace*{0.4 cm}  F_2 := x_n\circ m \hspace*{0.3 cm} \text{and}\hspace*{0.2 cm} F_3 :=  x_n\circ m_1
$$
are the linear first syzygies for the inverse system module. In addition, 
$$
x_n\circ F_1 - (b_i+1)x_1\circ F_2 + (b_1)x_i\circ F_3
$$
is a linear second syzygy for the inverse system module which contradicts the fact that the minimal free 
resolution of $I$ is linear for $n-2$ steps. Therefore, $y_n\mid m_1$ and since $y_n\nmid m$ we conclude that $i=n$.

By duality, if for $m\in \soc(S/I)$ we have $x_n\nmid m$, then there exist monomials  $m_1,\dots 
,m_{n-1}\in \soc(S/I)$ such that $\frac{m_i}{x_n}= \frac{m}{x_i}$ for all $1\leq i\leq n-1$. This implies that $x_nm\in \mathfrak{m}^{d+1}$ but  $x_nm\notin I$ which contradicts the assumption that
$\mathfrak{m}^{d+1}\subset I$. 
\end{proof}
Next example illustrates that the assumption, $\mathfrak{m}^{d+1}\subset I$ in Theorem \ref{n-2 steps} is necessary.
\begin{exmp}
The artinian monomial ideal $I=(x^3_1,x^3_2,x^3_3,x_1x_2x_3)$ in $S = \mathbb{K}[x_1,x_2,x_3]$ defines a \textit{Togliatti system} and therefore fails the WLP. Note that the minimal free resolution of $S/I$ is linear for $1$ step but $\m^4\nsubseteq I $.
\end{exmp}

\section{Acknowledgment}
We would like to thank Mats Boij who provided insight that greatly assisted this research. The first author was supported by the grant VR2013-4545. This work was done while the second author was visiting the Royal Institute of Technology (KTH) and he expresses his gratitude for  this hospitality.
\bibliographystyle{acm} 
\bibliography{bib.bib}

\end{document}